\documentclass{amsart}
\usepackage{amsfonts,amssymb}
\usepackage{enumerate,graphicx}
\usepackage{caption}
\usepackage{subcaption}
\usepackage[square,numbers]{natbib}
\usepackage{mathrsfs,bbm}
\usepackage{tikz,tikzscale,tikz-cd}
\usepackage{multirow,xparse,fp}
\usepackage{environ}
\usepackage{amstext}
\usepackage{hyperref}
\usepackage{diagbox}
 \usepackage[pagewise]{lineno}%\linenumbers
\usepackage{comment}
\newtheorem{theorem}{Theorem}[section]

\newtheorem{corollary}[theorem]{Corollary}
\newtheorem{lemma}[theorem]{Lemma}
\theoremstyle{definition}

\newtheorem{example}[theorem]{Example}

\newtheorem{remark}[theorem]{Remark}

\allowdisplaybreaks[4]%align可跨頁

\begin{document}
\title{Boundary complexity and surface entropy of 2-multiplicative integer systems on $\mathbb{N}^d$}

\author[Jung-Chao Ban]{Jung-Chao Ban}
\address[Jung-Chao Ban]{Department of Mathematical Sciences, National Chengchi University, Taipei 11605, Taiwan, ROC.}
\address{Math. Division, National Center for Theoretical Science, National Taiwan University, Taipei 10617, Taiwan. ROC.}
\email{jcban@nccu.edu.tw}

\author[Wen-Guei Hu]{Wen-Guei Hu}
\address[Wen-Guei Hu]{College of Mathematics, Sichuan University, Chengdu, 610064, China}
\email{wghu@scu.edu.cn}

\author[Guan-Yu Lai]{Guan-Yu Lai}
\address[Guan-Yu Lai]{Department of Applied Mathematics, National Yang Ming Chiao Tung University, Hsinchu 30010, Taiwan, ROC.}
\email{guanyu.am04g@g2.nctu.edu.tw}


\thanks{Ban is partially supported by the Ministry of Science and Technology, ROC (Contract MOST 111-2115-M-004-005-MY3). Hu is partially supported by the National Natural Science Foundation of China (Grant 11601355). Lai is partially supported by the Ministry of Science and Technology, ROC (Contract MOST 111-2811-M-004-002-MY2).}

%\date{}

\baselineskip=1.2\baselineskip

% -------------------------------------------------------------
\begin{abstract}
In this article, we introduce the concept of the boundary complexity and prove that for a 2-multiplicative integer system (2-MIS) $X^{p}_{\Omega}$ on $\mathbb{N}$ (or $X^{\bf p}_{\Omega}$ on $\mathbb{N}^d,d\geq 2$), every point in $[h(X^p_\Omega), \log r]$ can be realized as a boundary complexity of a 2-MIS with a specific speed, where r stands for the number of the alphabets. The result is new and quite different from $\mathbb{N}^d$ subshifts of finite type (SFT) for $d\geq 1$. Furthermore, the rigorous formula of surface entropy for a $\mathbb{N}^d$ 2-MIS is also presented. This provides an efficient method to calculate the topological entropy for $\mathbb{N}^d$ 2-MIS and also provides an intrinsic differences between $\mathbb{N}^d$ $k$-MIS and SFTs for $d\geq 1$ and $k\geq 2$.	
%In this article, we obtain that for any $h\in \lbrack h(X_{\Omega }^{\bf p}),\log r]$, it can be realized as a boundary complexity of 2-MIS $X_{\Omega }^{\bf p}$, $h^{\mathbf{f}}(X_{\Omega }^{\bf p})$, with a specific speed $\mathbf{f}$. And the boundary complexity of $\mathbb{N}^d$ SFTs are also computed. The surface entropy of 2-MIS is also stduied and obtain the sub exponential term of $|\mathcal{P}\left(\mathbb{Z}_{x_n,y_n},X^{\bf p}_{\Omega}\right)|$.  
\end{abstract}
\maketitle

% -------------------------------------------------------------
\tableofcontents
\section{Introduction}
\subsection{Multiplicative integer systems}

Let $\mathcal{A}$ be an alphabet and $\Omega \subseteq \mathcal{A}^{\mathbb{N}}$ be a subshift or a subset of $\mathcal{A}^{\mathbb{N}}$. For $\mathbf{p}_{1},\ldots ,\mathbf{p}_{k-1}\in \mathbb{N}^{d}$, $d\geq 1$ the \emph{$k$-multiplicative integer system} ($k$-MIS) with respect to $\Omega $ is defined as 
\begin{equation}
X_{\Omega }^{\mathbf{p}_{1},\ldots ,\mathbf{p}_{k-1}}=\{(x_{\mathbf{i}})_{\mathbf{i}\in \mathbb{N}^{d}}\in \mathcal{A}^{\mathbb{N}^{d}}:x_{\mathbf{i}}x_{\mathbf{ip}_{1}}\cdots x_{\mathbf{ip}_{k-1}}\in \Omega _{k}\text{ for
	all }\mathbf{i}\in \mathbb{N}^{d}\}\text{,}  \label{2}
\end{equation}
where $\Omega _{k}$ denotes the set of admissible blocks of length $1<k\in \mathbb{N}$. Suppose $T:X\rightarrow X$ is a map from $X$ to $X$, the study of (\ref{2}) motivated by the multifractal analysis of the following multiple
ergodic average (\ref{4}), which is a significant and active research topic
of the multiple ergodic theory (cf. \cite{fan2014some,
	kifer2010nonconventional, kifer2014nonconventional,
	furstenberg2014recurrence, furstenberg1982ergodic}). Namely, 
\begin{equation}
\frac{1}{n}S_{n}\mathbb{F}(x)=\frac{1}{n}
\sum_{k=1}^{n}f_{1}(T^{k}(x))f_{2}(T^{k}(x))\cdots f_{d}(T^{k}(x))\text{,}
\label{4}
\end{equation}
where $\mathbb{F}=(f_{1},f_{2},\ldots ,f_{d})$ is a $d$-tuple of functions
and $f_{i}:X\rightarrow \mathbb{R}$ for $1\leq i\leq d$. Fan et al. \cite{fan2012level} calculated the Hausdorff and Minkowski dimensions of (\ref{2}) for $(d,k)=(1,2)$ and $\Omega $ is an $1$-$d$ \emph{golden-mean shift} of $\{0,1\}^{\mathbb{N}}$, that is, the forbidden set of $\Omega $ is $\{11\}$.
The Minkowski dimension for general $k$-MIS is given in \cite{ban2019pattern}. Kenyon et al. \cite{kenyon2012hausdorff} extended the result of \cite{fan2012level} to general subshifts, and call
such a shift a \emph{multiplicative shift }because of the fact that it is invariant under the action of multiplicative integers, i.e., for $p\in 
\mathbb{N}$, $(x_{i})_{i=1}^{\infty }\in X_{\Omega }^{p}$ implies $(x_{ri})_{i=1}^{\infty }\in X_{\Omega }^{p}$ for all $r\in \mathbb{N}$.
Meanwhile, $X_{\Omega }^{p}$ is also called \emph{multiplicative SFT} if $\Omega $ is a SFT. Peres et al. \cite{peres2014dimensions} considered the set 
\begin{equation}
X_{\Omega }^{(S)}=\{x\in \{x_{k}\}_{k=1}^{\infty }\in \mathcal{A}^{\mathbb{N}}:x|_{iS}\in \Omega \mbox{ for all } i\text{ with }(i,s)=1\mbox{ for all }s\in S \},  \label{3}
\end{equation}
where $S=\langle p_{1},\ldots ,p_{j}\rangle $ is the semigroup generated by distinct primes $p_{1},\ldots ,p_{j}$. Both Hausdorff and Minkowski dimensions are computed there. As for the multifractal results of the multiple ergodic averages, Peres and Solomyak \cite{peres2012dimension}
established the Hausdorff dimension formula of the set 
\begin{equation}
E(\alpha )=\{\left( x_{k}\right) _{k=1}^{\infty }\in \{0,1\}^{\mathbb{N}}:\lim_{n\rightarrow \infty }\frac{1}{n}\sum_{k=1}^n x_{k}x_{2k}=\alpha \},\text{ }\alpha \in \lbrack 0,1].  \label{5}
\end{equation}
Fan, Schmeling and Wu \cite{fan2016multifractal} extended the result of \cite{peres2012dimension} to the general multiple ergodic averages, namely, the
multiple ergodic average of $\lim_{n\rightarrow \infty }\frac{1}{n}\sum_{k=1}^{n}x_{k}x_{2k}=\alpha $ in (\ref{5}) is replaced by the general
form 
\[
\lim_{n\rightarrow \infty }\frac{1}{n}\sum_{k=1}^{n}\phi
(x_{k},x_{kq},\ldots ,x_{kq^{l-1}})=\alpha,
\]
where $\phi:\mathcal{A}^l\mapsto \mathbb{R}$ is a function from $\mathcal{A}^l$ into $\mathbb{R}$. 

Suppose $\mathcal{A}=\{+1,-1\}$, Carinci et al. \cite{carinci2012nonconventional} and Chazottes and Redig \cite{chazottes2014thermodynamic} studied the large deviation principle and
thermodynamic limit of the free energy function associated by the multiple
sum $S_{n}(\sigma )=\sum_{k=1}^{n}\sigma _{k}\sigma _{2k}$ as a Hamiltonian
in the lattice spin system on $\mathcal{A}^{\mathbb{Z}}$. We emphasize that
this is the simplest version of the `multiplicative Ising model'. Such a Hamiltonian is therefore a long-range non-translation invariant interaction and thus much
more difficult to treat.

There are only few results for multidimensional multiplicative shifts (\ref{2}), i.e., $d\geq 2$. Ban et al.\cite{ban2021large} generalized the large
deviation result of \cite{carinci2012nonconventional} to $d$, $k\geq 2$. The
Minkowski dimension of (\ref{2}) for golden-mean shift $\Omega $ is
calculated in \cite{ban2021entropy}. Brunet \cite{brunet2021dimensions}
considered the set 
\[
X_{\Omega }^{[p]}=\{(x_{k},y_{k})_{k=1}^{\infty }\in \left( \mathcal{A}_{1}\times \mathcal{A}_{2}\right) ^{\mathbb{N}}:(x_{ip^{l}},y_{ip^{l}})_{l=0}^{\infty }\in \Omega \text{ }\forall i\text{
	with }p\nmid i\}\text{,}
\]%
where $\mathcal{A}_{i}=\{0,\ldots ,m_{i}-1\}$ for $i=1,2$ and $m_{1}\geq
m_{2}\geq 2$, and calculated the Hausdorff and Minkowski dimensions. The main
objective of this study is to demonstrate that the set $X_{\Omega }^{\mathbf{p}_{1},\ldots ,\mathbf{p}_{k-1}}$ presents a very different phenomena from
the $\mathbb{Z}^{d}$ SFTs for $d\geq 1$ in two aspects, namely, the `boundary complexity' and `surface entropy'.

\subsection{Entropy and boundary complexity}
The first part of this work focuses on the possible values of the \emph{boundary complexity}, say $h^{\mathbf{f}}$. Before providing the formal definition of $h^{\mathbf{f}}$, we should offer the motivation behind this study.
Let $(X,T)$ be a dynamical system, the study of the subsystems of $(X,T)$ is an interesting and fundamental research subject since it makes it possible to understand the richness of the subdynamics inside $(X,T)$. The famous Krieger embedding theorem \cite{krieger1982subsystems} indicates that if $(X,\sigma _{X})$ is a mixing $\mathbb{Z}$ SFT and $(Y,\sigma _{Y})$ is a $\mathbb{Z}$ subshift with no periodic points and $h(\sigma _{Y})<h(\sigma_{X})$, then $(Y,\sigma _{Y})$ is topologically conjugated to a subshift contained in $(X,\sigma _{X})$. Desai \cite{desai2006subsystem} has demonstrated that a $\mathbb{Z}^{d}$ SFT with positive entropy has rich SFT subsystems.
Precisely, she proved that a $\mathbb{Z}^{d}$ SFT $X$ such that $h(X)>0$, then the set $\{h(Y):Y\subset X$ and $Y$ is an SFT$\}$ is dense in $[0,h(Y)]$. Such a result has been recently extended to large classes of groups, i.e., amenable groups \cite{bland2022subsystem}. They have proved that if $X$ is a $G$ SFT, where $G$ is a countable amenable group, and $Y\subset X$ is any subsystem such that $h(Y)<h(X)$. Then $\{h(Z):Y\subset Z\subset X$ is an SFT$
\}$ is dense in $[h(Y),h(X)]$. In the same spirit of the aforementioned results,
the aim of this study is to see the complexity of a multiplicative integer
system at the border and the richness of such complexity.

Suppose $F\subseteq \mathbb{N}^{d}$ is a finite lattices. We define $\mathcal{P}(F,X)=\{\left( x_{\mathbf{i}}\right) _{\mathbf{i}\in F}\in 
\mathcal{A}^{F}:x\in X\}$, i.e., the \emph{projection} of $x\in X$ on $F$.
Let $\mathbf{f=}\left( f_{1},\ldots ,f_{d}\right) $ be a $d$-tuple of
functions and $f_{i}$ is a function from $\mathbb{N}$ to $\mathbb{N}$, $%
\forall i=1,\ldots ,d$. Denote $\mathbf{f}_{\mathbf{m}%
}=(f_{1}(m_{1}),f_{2}(m_{2}),\ldots ,f_{d}(m_{d}))$, where $\mathbf{m}%
=(m_{1},m_{2},\ldots ,m_{d})\in \mathbb{N}^{d}$ and $F_{\mathbf{m}}(\mathbf{f%
})=\mathbb{N}_{\mathbf{m}}\backslash \mathbb{N}_{\mathbf{f}_{\mathbf{m}%
}}\subset \mathbb{N}^{d}$. The \emph{boundary complexity }$h^{\mathbf{f}}(X)$
is defined as 
\begin{equation}
h^{\mathbf{f}}(X)=\lim_{\mathbf{m}\rightarrow \infty }\frac{\log \left\vert 
	\mathcal{P}(F_{\mathbf{m}}(\mathbf{f}),X)\right\vert }{\left\vert F_{\mathbf{%
			m}}(\mathbf{f})\right\vert }\text{,}  \label{1}
\end{equation}%
whenever the limit of (\ref{1}) exists and the notation $\mathbf{m}%
\rightarrow \infty $ means that $m_{i}\rightarrow \infty $ $\forall
i=1,\ldots ,d$. It is worth pointing out that if $f_{i}\equiv 0$ for all $%
i=1,\ldots ,d$, then $h^{\mathbf{f}}(X)=h(X)$, that is, it equals the usual
topological entropy of $X$. The purpose of this article is to investigate
the possible values of $h^{\mathbf{f}}(X^{\mathbf{p}}_{\Omega})$ and to provide a
connection between $h^{\mathbf{f}}(X^{\mathbf{p}}_{\Omega})$ and $h(X)$, where $X$ is
the $\mathbb{Z}^{d}$ SFT. Specificically, we consider the function $\mathbf{f%
}\longmapsto h^{\mathbf{f}}(X)$ and concentrate on the set 
\[
\mathcal{H}^{B}(X)=\{h^{\mathbf{f}}(X):\mathbf{f}\text{ is a }d\text{-tuple
	of functions from }\mathbb{N}\text{ into }\mathbb{N}\}\text{.}
\]%
In contrast to the works of Desai \cite{desai2006subsystem} and Bland et al. \cite{bland2022subsystem} as mentioned in the last paragraph, the main reason
for studying $\mathcal{H}^{B}(X)$ is as follows. Note that the number $%
\left\vert \mathcal{P}(F_{\mathbf{m}}(\mathbf{f}),X)\right\vert $ is the
possible number of patterns appearing on the `boundary' $F_{\mathbf{m}}(%
\mathbf{f})$, where $\mathbf{f}$ indicates the speed of the width of the
boundary. Such a problem demonstrates the richness of the growth rate of
patterns with respect to the width of the boundary (i.e., $\left\vert F_{%
	\mathbf{m}}(\mathbf{f})\right\vert $) according to the rule of $X^{\mathbf{p}%
}$.

Let $\mathcal{A}=\{0,\ldots ,r-1\}$ and $d=1$. Theorem \ref{thm real 1d} reveals that
every $h\in \lbrack h(X_{\Omega }^{p}),\log r]$ can be realized as an $h^{%
	\mathbf{f}}(X_{\Omega }^{p})$ with a specific speed $\mathbf{f}=f_{1}$,
namely, $h^{\mathbf{f}}(X^{p}_{\Omega})=h$ with certain $\tau=\lim_{m\rightarrow \infty
}f_{1}(m)/m$. In other words, this means $\mathcal{H}^{B}(X^{p}_{\Omega})=[h(X_{%
	\Omega }^{p}),\log r]$. The same result is also valid for $d\geq 2$ (Theorem \ref{thm real nd}). We emphasize that $G$ is an amenable group and $X$ is an $G$ SFT. If 
$\{F_{n}\}_{n\geq 1}$ is a \emph{F\o lner sequence} of $G$ (cf. \cite%
{ceccherini2010cellular}), i.e., $\lim_{n\rightarrow \infty }\frac{%
	\left\vert KF_{n}\bigtriangleup F_{n}\right\vert }{\left\vert
	F_{n}\right\vert }=0$ for every finite subset $K\subset G$, then the topological
entropy $h(X)$ can be calculated as $h(X)=\lim_{n\rightarrow \infty }\frac{%
	\log \left\vert \mathcal{P}(F_{n},X)\right\vert }{\left\vert
	F_{n}\right\vert }$, and is independent of the choice of the F\o lner
sequence $\{F_{n}\}_{n\geq 1}$. It is known that $\mathbb{Z}^{d}$ is an
amenable group and since $F_{\mathbf{m}}(\mathbf{f})$ has a larger chance of being a F\o lner sequence\footnote{%
	For instance, for $\mathbf{f}$ with the minimal width $\min
	\{m-f_{i}(m):1\leq i\leq d\}\rightarrow \infty $.}, the property of $%
\mathcal{H}^{B}(X)=[h(X),\log r]$ may not be true if $X$ is an SFT on $%
\mathbb{Z}^{d}$, $d\geq 2$ in general since $\log r$ might not be achieved. Nonetheless, Corollary \ref{thm add constant hbc 2d} shows that
for $i\in \mathbb{N}$, and every $h\in \lbrack \frac{1}{i}\log \lambda _{\mathbf{H}_{i}(X)},\frac{1}{i}\log \lambda _{\mathbf{V}_{i}(X)}]$, there
exists a $t\in [0,1]$ and $\mathbf{f}$ such that $h^{\mathbf{f}}(X)=h=\frac{%
	(1-t)}{i}\log \lambda _{\mathbf{V}_{i}(X)}+\frac{t}{i}\log \lambda _{\mathbf{%
		H}_{i}(X)}$. Here $\mathbf{H}_{i}(X)$ (resp. $\mathbf{V}_{i}(X)$) is the
transition matrix of the horizontal (resp. vertical) \emph{strip shift} $%
H_{i}(X)$ (resp. $V_{i}(X)$) defined on $\mathbb{N\times }\{1,\ldots ,n\}$ (we
refer the reader to \cite{pavlov2012approximating, ban2005patterns} for the
formal definition and properties of strip shifts therein). It is known that
the entropies of the strip shift of $H_{i}(X)$ (or $V_{i}(X)$) approximates
the topological entropy $h(X)$, that is, $\lim_{i\rightarrow \infty }\frac{%
	\log \lambda _{\mathbf{H}_{i}(X)}}{i}=\lim_{i\rightarrow \infty }\frac{\log
	\lambda _{\mathbf{V}_{i}(X)}}{i}=h(X)$. Combining this with Corollary \ref{thm add constant hbc 2d}
yields that $\overline{\mathcal{H}^{B}(X)}=\{h(X):X$ is an SFT on $\mathbb{N}^{2}\}$, which provides a profound connection to the possible values of $h^{%
	\mathbf{f}}(X)$ and $h(X)$. It is worth pointing out that the speed function $%
\mathbf{f}=(f_{1},f_{2})$ in Corollary \ref{thm add constant hbc 2d} makes $F_{\mathbf{m}}(\mathbf{f})$
a constant width, which is different from the cases of multiplicative
integer systems described in Theorem \ref{thm real 1d} (or Theorem \ref{thm real nd}). Moreover, the general version for $\mathbb{N}^d$ is obtained in Theorem \ref{thm Nd sft bd}.

\subsection{Surface entropy}

The second part of this work concerns the surface entropy of $X^{\mathbf{p}}_\Omega$ on $\mathbb{N}^{2}$. Let $X$ be a subshift, the \emph{surface
	entropy} of $X$ with eccentricity $\alpha $ is 
\[
h_{s}(X,\alpha )=\sup_{\{(x_{n},y_{n})\}\in \Gamma _{\alpha
}}\limsup\limits_{n\rightarrow \infty }S_{X}(x_{n},y_{n})\text{,}
\]%
where 
\[
S_{X}(x_{n},y_{n})=\frac{\log \left\vert \mathcal{P}(\mathbb{Z}%
	_{x_{n},y_{n}},X)\right\vert -x_{n}y_{n}h(X)}{x_{n}+y_{n}},
\]%
and $\Gamma _{\alpha }=\{\{\left( x_{n},y_{n}\right) \}\in \left( \mathbb{N}%
^{2}\right) ^{\mathbb{N}}:\frac{y_{n}}{x_{n}}\rightarrow \alpha $ and $%
x_{n}\rightarrow \infty \}$. The investigation of the surface entropy is to
look at the `linear term' of the complexity function $\log \left\vert 
\mathcal{P}(\mathbb{Z}_{x_{n},y_{n}},X)\right\vert $. Pace \cite%
{pace2018surface} first introduced the concept of the surface entropy and
obtained the explicit formula for $\mathbb{Z}$ SFTs and many
interesting properties for $\mathbb{Z}^{2}$ SFTs. The characterization of
the possible surface entropies of $\mathbb{Z}^{2}$ SFTs (or $\mathbb{Z}^{2}$ sofic subshifts) as the $\Pi _{3}$ real numbers of $[0,+\infty ]$ is
presented in \cite{callard2021computational}. A deep link between the
surface entropy and entropy dimension is given in \cite{meyerovitch2011growth}. Let $X_{\Omega }^{\mathbf{p}}$ be a 2-MIS on $\mathbb{N}^{2}$, Theorem \ref{theorem main} (1) provides a rigorous
formula for $\log \left\vert \mathcal{P}(\mathbb{Z}_{x_{n},y_{n}},X_{\Omega
}^{\mathbf{p}})\right\vert -x_{n}y_{n}h(X_{\Omega }^{\mathbf{p}})$. Using
this result, we could provide a rigorous formula for the complexity fnction $%
\log \left\vert \mathcal{P}(\mathbb{Z}_{x_{n},y_{n}},X_{\Omega }^{\mathbf{p}%
})\right\vert $ with specific $x_{n}$ and $y_{n}$ (Theorem \ref{theorem main} (2) and (3)).
Surprisingly, if $x_{n}=y_{n}=p^{n}\pm k$ with $1\leq k\leq p$, $\log
\left\vert \mathcal{P}(\mathbb{Z}_{x_{n},y_{n}},X_{\Omega }^{\mathbf{p}%
})\right\vert =x_{n}y_{n}h(X_{\Omega }^{\mathbf{p}})+O(n)$ (Theorem \ref{theorem main}
(3)), this provides us with an efficient and accurate estimate of the growth rate of
patterns $\left\vert \mathcal{P}(\mathbb{Z}_{x_{n},y_{n}},X_{\Omega }^{%
	\mathbf{p}})\right\vert $ and this result is different than $\mathbb{Z}^{2}$ SFTs. Finally, the $\mathbb{N}^{d}$ version of Theorem \ref{theorem main} is presented in
Theorem \ref{thm 2.1 nd}.

\section{Boundary complexity}
\subsection{Preliminaries}
Define the 2-multiplicative integer system on $\mathbb{N}^d$ by
\begin{equation*}
	X^{\bf p}_{\Omega}=\left\{ \left(x_{\bf i}\right)_{{\bf i}\in\mathbb{N}^d} \in \left\{0,1,...,r-1\right\}^{\mathbb{N}^d}: \left(x_{{\bf i} \cdot {\bf p}^{\ell} }\right)_{\ell=0}^\infty\in \Omega \mbox{ for all }{\bf i} \in \mathbb{N}^d \right\},
\end{equation*}
where ${\bf p}=\left(p_1,...,p_d \right)$ and $\Omega$ is a subshift. Given $p_1,...,p_d\geq2$ and $N_1,...,N_d\geq 1$, we let $\mathcal{M}_{{\bf p}}=\left\{ (p_1^m,...,p_d^m) :m\geq 0\right\}$ be the subset of $\mathbb{N}^d$, and denote by $\mathcal{M}_{\bf p}({\bf i})$ a version of the lattice $\mathcal{M}_{\bf p}$ starting from ${\bf i}\in \mathbb{N}^d$, i.e. $\mathcal{M}_{{\bf p}}({\bf i})=\left\{ (i_1p_1^m,...,i_dp_d^m) :m\geq 0\right\}$. Finally we define $\mathcal{I}_{{\bf p}}=\left\{ {\bf i}\in \mathbb{N}^d : p_j\nmid i_j\mbox{ for some }1\leq j \leq d \right\}$ as the index set of $\mathbb{N}^d$.

More definitions are needed to characterize the partition of the $N_1\times \cdots \times N_d$ lattice. Let $\mathbb{Z}_{N_1\times \cdots \times N_d}=\left\{ {\bf i}\in \mathbb{N}^d : 1\leq i_j\leq N_j \mbox{ for all } 1\leq j \leq d  \right\}$ be the $N_1\times \cdots \times N_d$ lattice and $\mathcal{L}_{N_1\times \cdots \times N_d}({\bf i})= \mathcal{M}_{{\bf p}}({\bf i})\cap \mathbb{Z}_{N_1\times \cdots \times N_d}$ be the subset of $\mathcal{M}_{\bf p}({\bf i}) $ in the $N_1\times \cdots \times N_d$ lattice. Then we define $\mathcal{J}_{N_1\times \cdots \times N_d;\ell}= \left\{ {\bf i}\in \mathbb{Z}_{N_1\times \cdots \times N_d} : \left|\mathcal{L}_{N_1\times \cdots \times N_d}({\bf i})\right|=\ell  \right\}$, where $| \cdot |$ denotes cardinality, as the set of points ${\bf i}$ in the $N_1\times \cdots \times N_d$ lattice such that the cardinality of the set $\mathcal{M}_{\bf p}({\bf i})\cap \mathbb{Z}_{N_1\times \cdots \times N_d}$ is exactly $\ell$. Let $\mathcal{K}_{N_1\times \cdots \times N_d;\ell}=\{ {\bf i}\in \mathcal{I}_{{\bf p}}\cap \mathbb{Z}_{N_1\times \cdots \times N_d} : |\mathcal{L}_{N_1\times \cdots \times N_d}({\bf i})|=\ell  \}$ be the set of points ${\bf i}$ in $\mathcal{I}_{\bf p}$ so that the cardinality of the set $\mathcal{M}_{\bf p}({\bf i})\cap \mathbb{Z}_{N_1\times \cdots \times N_d}$ is exactly $\ell$. The following lemmas give the disjointed decomposition of $\mathbb{N}^d$ and the limit of the density of $\mathcal{K}_{N_1\times \cdots \times N_d;\ell}$ which are the $\mathbb{N}^d$ version of Lemmas 2.1 and 2.2 \cite{ban2021entropy}, respectively.

\begin{lemma}\label{lemma:2.2}
	For $p_1,...,p_d\geq2$,
	\begin{equation*}
		\mathbb{N}^d=\displaystyle\bigsqcup\limits_{{\bf i} \in \mathcal{I}_{{\bf p}}}\mathcal{M}_{{\bf p}}({\bf i}).
	\end{equation*}  
\end{lemma}

\begin{lemma}\label{lemma decompose2}
	For $N_1,...,N_d$, and $\ell \geq 1$, we have the following assertions.	
	\item[\bf1.] $\left|\mathcal{J}_{N_1\times N_2\times \cdots \times N_d;\ell}\right|= \prod_{k=1}^{d}\left\lfloor\frac{N_k}{p_k^{\ell-1}}\right\rfloor-\prod_{k=1}^{d}\left\lfloor\frac{N_k}{p_k^\ell}\right\rfloor.$
	\item[\bf2.] $\lim_{N_1,...,N_d\to \infty}\frac{\left| \mathcal{K}_{N_1\times \cdots \times N_d;\ell} \right| }{ \left| \mathcal{J}_{N_1\times \cdots \times N_d;\ell} \right| }=1-\frac{1}{p_1\cdots p_d}.$
	\item[\bf3.] $\lim_{N_1,...,N_d\to \infty}\frac{| \mathcal{K}_{N_1\times \cdots \times N_d;\ell} | }{ N_1\cdots N_d }=\frac{\left(p_1\cdots p_d-1\right)^2}{\left(p_1\cdots p_d\right)^{\ell+1}}.$	
	\item[\bf4.] $\lim_{N_1,...,N_d\to \infty}\sum_{\ell=1}^{N_1\cdots N_d}\frac{\left|\mathcal{K}_{N_1 \times \cdots \times N_d;\ell}\right|\log |\Omega_{\ell}|}{N_1\cdots N_d}
	=\sum_{\ell=1}^{\infty}\lim_{N_1,...,N_d\to \infty}\frac{\left|\mathcal{K}_{N_1\times \cdots \times N_d;\ell}\right|\log |\Omega_{\ell}|}{N_1\cdots N_d}.$		
\end{lemma}

\subsection{The boundary complexity of the 2-MIS on $\mathbb{N}^d$}
In this section, we prove the dense property of the boundary complexity for a 2-MIS on $\mathbb{N}$ (Theorem \ref{thm real 1d}) and $\mathbb{N}^d$ (Theorem \ref{thm real nd}) by establishing the corresponding rigorous formula for boundary complexities (Lemma \ref{lemma hbc1d} for $\mathbb{N}$ and Lemma \ref{lemma hbc2d} for general $\mathbb{N}^d$).

\subsubsection{$d=1$}
Lemma \ref{lemma hbc1d} presents the rigorous formula for a 2-MIS on $\mathbb{N}$.
\begin{lemma}\label{lemma hbc1d}
	For $\ell \geq 1$, if $\frac{m}{p^\ell}< f_1(m) \leq \frac{m}{p^{\ell-1}}$, we have
	\begin{equation*}
			h^{\bf f}\left(X^p_{\Omega} \right)=\left(\frac{1}{1-\tau}\right)\left[\left(1-\frac{1}{p}\right)\zeta_\ell +\eta_\ell	 \right],	
	\end{equation*}
	where $\tau=\lim_{m\to \infty}\frac{f_1(m)}{m}$,
	\[\zeta_\ell=\left(p-1\right)\sum_{i=1}^{\ell-2}\frac{\log |\Omega_i|}{p^i}+\left(\frac{1}{p^{\ell-2}}-p\tau\right)\log |\Omega_{\ell-1}|\]
	and \[\eta_\ell=\left(p\tau-\frac{1}{p^{\ell-1}}\right)\log |\Omega_{\ell-1}|+\left(\frac{1}{p^{\ell-1}}-\tau\right)\log |\Omega_{\ell}|.\]
\end{lemma}

\begin{proof}
	The proof of the general cases, i.e., $d\geq1$ is presented in Lemma \ref{lemma hbc2d}, thus we omit it. 
\end{proof}

Using Lemma \ref{lemma hbc1d}, the dense property of the boundary complexity for 2-multiplicative integer system on $\mathbb{N}$ is presented.

\begin{theorem}\label{thm real 1d}
	For any $h\in \left[ h\left(X^p_{\Omega} \right),\log |\Omega_1|\right]$, there is a $\tau$ such that $h$ can be realized by boundary complexity $h^{\bf f}$. That is, $h^{\bf f}\left(X^p_{\Omega} \right)=h$.
\end{theorem}
\begin{proof}
	For $h=h\left(X^p_{\Omega} \right)$, let $f_1(m)=1$ for all $m\geq 1$. Then we have $\tau=0$ and $h^{\bf f}\left(X^p_{\Omega} \right)=h\left(X^p_{\Omega} \right)$. 
	
	For $h\in \left(  \frac{\left(1-\frac{1}{p}\right)\log |\Omega_1|+\frac{1}{p}\log |\Omega_2|}{1+\frac{1}{p}},\log |\Omega_1|\right]$, let 
	\begin{equation*}
	\frac{1}{p^2}<\tau=\frac{\frac{1}{p}\log |\Omega_2|+\left(1-\frac{2}{p}\right)\log |\Omega_1| -h }{\log |\Omega_2|-\log |\Omega_1|-h}\leq \frac{1}{p}.
	\end{equation*}
Then we have $h^{\bf f}\left(X^p_{\Omega} \right)=h$ by applying Lemma \ref{lemma hbc1d} with $\ell=2$. In general, if $h\in \left(A_{k+1},A_k\right]$,
where
\begin{equation*}
A_k=\left(\frac{1}{1-\frac{1}{p^{k-1}}}\right)\left\{\left(1-\frac{1}{p}\right)\left[\left(p-1\right)\sum_{i=1}^{k-2}\frac{\log |\Omega_i|}{p^i}\right]
+\left(\frac{1}{p^{k-2}}-\frac{1}{p^{k-1}}\right)\log |\Omega_{k-1}| \right\}.
\end{equation*}
Note that $k=2$ is above case and $\bigcup_{k=2}^\infty \left(A_{k+1},A_k\right]=\left(  h\left(X^p_{\Omega} \right),\log |\Omega_1|\right]$. For this case, let
\begin{equation*}
\tau=\frac{\left(1-\frac{1}{p}\right)\left[\left(p-1\right)\sum_{i=1}^{k-2}\frac{\log |\Omega_i|}{p^i}\right]+\left(\frac{1}{p^{k-2}}-\frac{2}{p^{k-1}}\right)\log |\Omega_{k-1}|+\frac{1}{p^{k-1}}\log |\Omega_{k}|-h}{\log |\Omega_{k}|-\log |\Omega_{k-1}|-h},
\end{equation*}
then $\frac{1}{p^k}<\tau\leq \frac{1}{p^{k-1}}$ and thus the proof is complete by Lemma \ref{lemma hbc1d} with $\ell=k$.
\end{proof}

\subsubsection{$d\geq 2$}
This section shows the dense property of 2-MIS on $\mathbb{N}^d$ for $d\geq 2$. For the readers’ convenience, we provide a complete proof for $d=2$, and omit the proof for $d>2$ since the proof is almost identical. Lemma \ref{lemma hbc2d} below provides the explicit formula for the boundary complexity $h^{\bf f}$ of 2-MIS on $\mathbb{N}^2$ with a specific speed ${\bf f}=(f_1, f_2)$.

\begin{lemma}\label{lemma hbc2d}
	For $\ell \geq 1$, if $\frac{m}{p_1^\ell}< f_1(m) \leq \frac{m}{p_1^{\ell-1}}$ and $\frac{n}{p_2^\ell}< f_2(n) \leq \frac{n}{p_2^{\ell-1}}$, we have
	\begin{equation*}
	h^{\bf f}\left(X^{\bf p}_{\Omega} \right)=\left(\frac{1}{1-\tau}\right)\left[\left(1-\frac{1}{p_1p_2}\right)\zeta_\ell+\eta_\ell	 \right],	
\end{equation*}
where $\tau=\lim_{m,n\to \infty}\frac{f_1(m)f_2(n)}{mn}$,
 \[\zeta_\ell=\left(p_1p_2-1\right)\sum_{i=1}^{\ell-2}\frac{\log |\Omega_i|}{\left(p_1p_2\right)^i}+\left(\frac{1}{\left(p_1p_2\right)^{\ell-2}}-p_1p_2\tau\right)\log |\Omega_{\ell-1}|\]and \[\eta_\ell=\left(p_1p_2\tau-\frac{1}{\left(p_1p_2\right)^{\ell-1}}\right)\log |\Omega_{\ell-1}|+\left(\frac{1}{\left(p_1p_2\right)^{\ell-1}}-\tau\right)\log |\Omega_{\ell}|.\]
\end{lemma}

\begin{proof}
	The proof is divided into three parts, namely, $\ell=1$, $\ell=2$, and $\ell\geq 3$. For $\ell=1$, all the boundary points in this area belong to different independent layers. For $\ell=2$, the boundary points are divided into 3 parts: the area containing the independent layers with cardinality 1, i.e. $\mathbb{Z}_{p_1f_1(m)\times p_2f_2(n)}\setminus\mathbb{Z}_{\left\lfloor\frac{m}{p_1}\right\rfloor\times \left\lfloor\frac{n}{p_2}\right\rfloor}$, the area containing layers that intersect $ \mathbb{Z}_{\left\lfloor\frac{m}{p_1}\right\rfloor\times \left\lfloor\frac{n}{p_2}\right\rfloor}\setminus \mathbb{Z}_{f_1(m)\times f_2(n)}$ nonempty, and the area containing the independent layers with cardinality 2, i.e., $ \mathbb{Z}_{\left\lfloor\frac{m}{p_1}\right\rfloor\times \left\lfloor\frac{n}{p_2}\right\rfloor}\setminus \mathbb{Z}_{f_1(m)\times f_2(n)}$. For $\ell\geq 3$, the boundary area will be divided into 4 parts: the area containing independent layers with cardinality $\ell$, i.e. $ \mathbb{Z}_{\left\lfloor\frac{m}{p_1^{\ell-1}}\right\rfloor\times \left\lfloor\frac{n}{p_2^{\ell-1}}\right\rfloor}\setminus \mathbb{Z}_{f_1(m)\times f_2(n)}$, the area containing independent layers with cardinality $\ell-1$ and that intersect $\mathbb{Z}_{\left\lfloor\frac{m}{p_1^{\ell-1}}\right\rfloor\times \left\lfloor\frac{n}{p_2^{\ell-1}}\right\rfloor}\setminus \mathbb{Z}_{f_1(m)\times f_2(n)}$ nonempty, i.e. $\mathbb{Z}_{\left\lfloor\frac{m}{p_1^{\ell-2}}\right\rfloor\times \left\lfloor\frac{n}{p_2^{\ell-2}}\right\rfloor}\setminus\mathbb{Z}_{p_1f_1(m)\times p_2f_2(n)}$, the area containing independent layers with cardinality only $\ell-1$, i.e. $\mathbb{Z}_{p_1f_1(m)\times p_2f_2(n)}\setminus\mathbb{Z}_{\left\lfloor\frac{m}{p_1^{\ell-1}}\right\rfloor\times \left\lfloor\frac{n}{p_2^{\ell-1}}\right\rfloor}$, and the remainder area containing the independent layers with cardinality $i, 1\leq i \leq \ell-2$.       
	\item[\bf1.] If $\ell=1$ and $1\geq \tau >\frac{1}{p_1p_2}$, then we have 
	\begin{equation*}
	\log \left|\mathcal{P}\left(F_{\bf m}({\bf f}),X^{\bf p}_{\Omega}\right)\right|=\log |\Omega_1|^{\left| \mathbb{Z}_{m\times n}\setminus \mathbb{Z}_{f_1(m)\times f_2(n)}\right|}.
	\end{equation*} 
This implies $h^{\bf f}\left(X^{\bf p}_{\Omega} \right)=\log |\Omega_1|=\log r$.
	\item[\bf2.] If $\ell=2$ and $\frac{1}{p_1p_2}\geq \tau >\frac{1}{\left(p_1p_2\right)^2}$, then we have
	\begin{equation*}
		\begin{aligned}
		&\log \left|\mathcal{P}\left(F_{\bf m}({\bf f}),X^{\bf p}_{\Omega}\right)\right|\\
		&=\log |\Omega_1|^{\left| \mathbb{Z}_{m\times n}\setminus \mathbb{Z}_{p_1f_1(m)\times p_2f_2(n)}\right|\frac{\left| \mathcal{K}_{m\times n;1}\setminus \mathbb{Z}_{p_1f_1(m)\times p_2f_2(n)} \right|}{\left| \mathbb{Z}_{m\times n}\setminus \mathbb{Z}_{p_1f_1(m)\times p_2f_2(n)}\right|}}\\
		&+\log |\Omega_1|^{\left|  \mathbb{Z}_{p_1f_1(m)\times p_2f_2(n)}\setminus\mathbb{Z}_{\left\lfloor\frac{m}{p_1}\right\rfloor\times \left\lfloor\frac{n}{p_2}\right\rfloor}\right|}+\log |\Omega_2|^{\left| \mathbb{Z}_{\left\lfloor\frac{m}{p_1}\right\rfloor\times \left\lfloor\frac{n}{p_2}\right\rfloor}\setminus \mathbb{Z}_{f_1(m)\times f_2(n)}\right|}.
		\end{aligned}
	\end{equation*}

Since $m-p_1f_1(m)\to\infty$ and $n-p_2f_2(n)\to\infty$ as $m,n\to\infty$, then by the similar proof of Lemma \ref{lemma decompose2}, we have 
\begin{equation*}
\lim_{n\to\infty}\frac{\left| \mathcal{K}_{m\times n;1}\setminus \mathbb{Z}_{p_1f_1(m)\times p_2f_2(n)} \right|}{\left| \mathbb{Z}_{m\times n}\setminus \mathbb{Z}_{p_1f_1(m)\times p_2f_2(n)}\right|}=1-\frac{1}{p_1p_2}.
\end{equation*}

This implies
\begin{equation*}
h^{\bf f}\left(X^{\bf p}_{\Omega} \right)
=\frac{\left(1-p_1p_2\tau\right)\left(1-\frac{1}{p_1p_2}\right)}{1-\tau}\log |\Omega_1|+\frac{\left(p_1p_2\tau-\frac{1}{p_1p_2}\right)}{1-\tau}\log |\Omega_1|+\frac{\left(\frac{1}{p_1p_2}-\tau\right)}{1-\tau}\log |\Omega_2|.	
\end{equation*}
\item[\bf3.] If $\ell\geq 3$ and $\frac{1}{\left(p_1p_2\right)^{\ell-1}}\geq \tau >\frac{1}{\left(p_1p_2\right)^{\ell}}$, then we have 
\begin{equation*}
	\begin{aligned}
		&\log \left|\mathcal{P}\left(F_{\bf m}({\bf f}),X^{\bf p}_{\Omega}\right)\right|\\
		&=\sum_{i=1}^{\ell-2}\log |\Omega_i|^{\left| \mathbb{Z}_{\left\lfloor m/p_1^{i-1}\right\rfloor\times \left\lfloor n/p_2^{i-1}\right\rfloor}\setminus \mathbb{Z}_{\left\lfloor m/p_1^i\right\rfloor\times \left\lfloor n/p_2^i\right\rfloor}\right|\frac{\left| \mathcal{K}_{\left\lfloor m/p_1^{i-1}\right\rfloor\times \left\lfloor n/p_2^{i-1}\right\rfloor;i}\setminus \mathbb{Z}_{\left\lfloor m/p_1^i\right\rfloor\times \left\lfloor n/p_2^i\right\rfloor}\right|}{\left| \mathbb{Z}_{\left\lfloor m/p_1^{i-1}\right\rfloor\times \left\lfloor n/p_2^{i-1}\right\rfloor}\setminus \mathbb{Z}_{\left\lfloor m/p_1^i\right\rfloor\times \left\lfloor n/p_2^i\right\rfloor}\right|}}\\
		&+\log |\Omega_{\ell-1}|^{\left|\mathbb{Z}_{\left\lfloor m/p_1^{\ell-2}\right\rfloor\times \left\lfloor n/p_2^{\ell-2}\right\rfloor}\setminus\mathbb{Z}_{p_1f_1(m)\times p_2f_2(n)}\right|\frac{\left|\mathcal{K}_{\left\lfloor m/p_1^{\ell-2}\right\rfloor\times \left\lfloor n/p_2^{\ell-2}\right\rfloor;\ell-1}\setminus\mathbb{Z}_{p_1f_1(m)\times p_2f_2(n)}\right|}{\left|\mathbb{Z}_{\left\lfloor m/p_1^{\ell-2}\right\rfloor\times \left\lfloor n/p_2^{\ell-2}\right\rfloor}\setminus\mathbb{Z}_{p_1f_1(m)\times p_2f_2(n)}\right|}}\\
		&+\log |\Omega_{\ell-1}|^{\left|\mathbb{Z}_{p_1f_1(m)\times p_2f_2(n)}\setminus\mathbb{Z}_{\left\lfloor m/p_1^{\ell-1}\right\rfloor\times \left\lfloor n/p_2^{\ell-1}\right\rfloor}\right|}+\log |\Omega_{\ell}|^{\left| \mathbb{Z}_{\left\lfloor m/p_1^{\ell-1}\right\rfloor\times \left\lfloor n/p_2^{\ell-1}\right\rfloor}\setminus \mathbb{Z}_{f_1(m)\times f_2(n)}\right|}.
	\end{aligned}
\end{equation*}	

Since for $1\leq i \leq \ell-2$, $\left\lfloor\frac{m}{p_1^{i-1}}\right\rfloor-\left\lfloor\frac{m}{p_1^i}\right\rfloor,\left\lfloor\frac{n}{p_2^{i-1}}\right\rfloor-\left\lfloor\frac{n}{p_2^i}\right\rfloor,\left\lfloor\frac{m}{p_1^{\ell-2}}\right\rfloor-p_1f_1(m)$ and $\left\lfloor\frac{n}{p_2^{\ell-2}}\right\rfloor-p_2f_2(n)\to\infty$ as $m,n\to \infty$. Then by the same reason above, we have
\begin{equation*}
\lim_{m,n\to\infty}\frac{\left| \mathcal{K}_{\left\lfloor m/p_1^{i-1}\right\rfloor\times \left\lfloor n/p_2^{i-1}\right\rfloor;i}\setminus \mathbb{Z}_{\left\lfloor m/p_1^i\right\rfloor\times \left\lfloor n/p_2^i\right\rfloor}\right|}{\left| \mathbb{Z}_{\left\lfloor m/p_1^{i-1}\right\rfloor\times \left\lfloor n/p_2^{i-1}\right\rfloor}\setminus \mathbb{Z}_{\left\lfloor m/p_1^i\right\rfloor\times \left\lfloor n/p_2^i\right\rfloor}\right|}=1-\frac{1}{p_1p_2},
\end{equation*}
and
\begin{equation*}
\lim_{m,n\to\infty}\frac{\left|\mathcal{K}_{\left\lfloor m/p_1^{\ell-2}\right\rfloor\times \left\lfloor n/p_2^{\ell-2}\right\rfloor;\ell-1}\setminus\mathbb{Z}_{p_1f_1(m)\times p_2f_2(n)}\right|}{\left|\mathbb{Z}_{\left\lfloor m/p_1^{\ell-2}\right\rfloor\times \left\lfloor n/p_2^{\ell-2}\right\rfloor}\setminus\mathbb{Z}_{p_1f_1(m)\times p_2f_2(n)}\right|}=1-\frac{1}{p_1p_2}.
\end{equation*}

This implies
\begin{equation*}
	\begin{aligned}
		h^{\bf f}\left(X^{\bf p}_{\Omega} \right)
		&=\sum_{i=1}^{\ell-2}\frac{\left(\frac{1}{\left(p_1p_2\right)^{i-1}}-\frac{1}{\left(p_1p_2\right)^i}\right)\left(1-\frac{1}{p_1p_2}\right)}{1-\tau}\log |\Omega_i|\\
		&+\frac{\left(\frac{1}{\left(p_1p_2\right)^{\ell-2}}-p_1p_2\tau\right)\left(1-\frac{1}{p_1p_2}\right)}{1-\tau}\log |\Omega_{\ell-1}|\\
		&+\frac{\left(p_1p_2\tau-\frac{1}{\left(p_1p_2\right)^{\ell-1}}\right)}{1-\tau}\log |\Omega_{\ell-1}|+\frac{\left(\frac{1}{\left(p_1p_2\right)^{\ell-1}}-\tau\right)}{1-\tau}\log |\Omega_{\ell}|.
	\end{aligned}	
\end{equation*}
The proof is now complete.
\end{proof}

Lemma \ref{lemma d>=2} has a similar result to Lemma \ref{lemma hbc2d} for multiplicative integer system on $\mathbb{N}^d,d\geq 2$.

\begin{lemma}\label{lemma d>=2}
	For $\ell \geq 1$, if $\frac{n_i}{p_i^\ell}<f_i(n_i)\leq \frac{n_i}{p_i^{\ell-1}}$ for all $1\leq i \leq d$, we have
	\begin{equation*}
			h^{\bf f}\left(X^{\bf p}_{\Omega} \right)=\left(\frac{1}{1-\tau}\right)\left[\left(1-\frac{1}{p_1\cdots p_d}\right)\zeta_\ell+\eta_\ell\right],	
	\end{equation*}
	where $\tau=\lim_{n_1,...,n_d\to \infty}\frac{f_1(n_1)\cdots f_d(n_d)}{n_1\cdots n_d}$,
	 \[\zeta_\ell=\sum_{i=1}^{\ell-2}\frac{\left(p_1\cdots p_d-1\right)\log |\Omega_i|}{\left(p_1\cdots p_d\right)^i}+\left(\frac{1}{\left(p_1\cdots p_d\right)^{\ell-2}}-p_1\cdots p_d\tau\right)\log |\Omega_{\ell-1}|\] and 
	 \[\eta_\ell=\left(p_1\cdots p_d\tau-\frac{1}{\left(p_1\cdots p_d\right)^{\ell-1}}\right)\log |\Omega_{\ell-1}|+\left(\frac{1}{\left(p_1\cdots p_d\right)^{\ell-1}}-\tau\right)\log |\Omega_{\ell}|.\] 
\end{lemma}

\begin{proof}
	The proof is almost identical to the proof of Lemma \ref{lemma hbc2d}, so we omit it.
\end{proof}

Using Lemma \ref{lemma d>=2} with the same argument as Theorem \ref{thm real 1d}, we obtain the dense property for 2-MIS on $\mathbb{N}^d,d\geq 2$. 

\begin{theorem}\label{thm real nd}
	For any $h\in \left[ h\left(X^{\bf p}_{\Omega} \right),\log |\Omega_1|\right]$, there is a $\tau$ such that $h$ can be realized by boundary complexity $h^{\bf f}$. That is, $h^{\bf f}\left(X^{\bf p}_{\Omega} \right)=h$.
\end{theorem}

\subsection{The boundary complexity for $\mathbb{N}^d$ SFTs}
Given $(i_{1},i_{2},\ldots, i_{d})\in \mathbb{N}^{d}$, for $n_{j}\geq 1$ (or $n_{j}=\infty$), $1\leq j\leq d$, define the $n_{1}\times n_{2}\times\cdots \times n_{d}$ rectangular lattice with initial point $(i_{1},i_{2},\ldots, i_{d})$ by
\begin{equation*}
	\mathbb{N}_{n_{1}\times n_{2}\times\cdots \times n_{d}}(i_{1},i_{2},\ldots, i_{d})=[i_{1},i_{1}+n_{1}-1]\times[i_{2},i_{2}+n_{2}-1]\times\cdots \times [i_{d},i_{d}+n_{d}-1],
\end{equation*}
where $[m_{1},m_{2}]$ means the set of integers between $m_{1}$ and $m_{2}$. In particular, let $\mathbb{N}_{n_{1}\times n_{2}\times\cdots \times n_{d}}=\mathbb{N}_{n_{1}\times n_{2}\times\cdots \times n_{d}}(1,1,\ldots,1)$.
Suppose $X\subseteq \mathcal{A}^{\mathbb{N}^{d}}$ is an $\mathbb{N}^{d}$ SFTs with finite forbidden set $\mathcal{F}$.  For $1\leq k\leq d$ and $i\geq 1$, the strip shifts $\Sigma^{(k)}_{i}(X)$  with thickness $i$ in $k$th-direction is defined by
\begin{equation*}
	\Sigma^{(k)}_{i}(X)=\left\{u\in \mathcal{A}^{\mathbb{N}_{n_{1}\times n_{2}\times\cdots \times n_{d}}} :  u \text{ contains no forbidden patterns in }\mathcal{F}\right\},\footnote{For $d=2$, we simply write $\Sigma_i^{(1)}(X)=V_i(X)$ and $\Sigma_i^{(2)}(X)=H_i(X)$, which are mentioned in the introduction.}
\end{equation*}
where  $n_{k}=i$ and $n_{j}=\infty$ for $j \neq k$. It is well-known that $\Sigma^{(k)}_{i}(X)$ can be regarded as a $(d-1)$-dimensional SFT (with thickness $i$). Then, we define the entropy $h^{(k)}_{i}(X)$ of $\Sigma^{(k)}_{i}(X)$ by
\begin{equation*}
	h^{(k)}_{i}(X)=\underset{\underset{j\neq k }{n_{j}\rightarrow \infty}}{\lim} \frac{\log \mathcal{P}(\mathbb{N}_{n_{1}\times\cdots \times n_{k-1} \times i \times n_{k+1}\times\cdots \times n_{d}},X) }{n_{1}\cdot\ldots\cdot n_{k-1}\cdot n_{k+1}\ldots \cdot n_{d}},
\end{equation*}
where for $F\subseteq \mathbb{N}^{d}$, $\mathcal{P}(F,X)$ is the cardinality of $\left\{x|_{F}:x\in X\right\}$. Clearly, 
	\begin{align*}
		h_{top}(X)&\equiv  \lim_{n_j\to\infty,1\leq j\leq d}\frac{\log \mathcal{P}(\mathbb{N}_{n_{1}\times n_{2}\times\cdots \times n_{d}},X) }{n_{1}\cdot n_{2}\cdot\ldots \cdot n_{d}} \\
		&= \lim_{i\to \infty} \frac{h^{(k)}_{i}(X)}{i},
	\end{align*}
see Pavlov \cite{pavlov2012approximating}.

For a $\mathbb{N}^{d}$ SFT $X\subseteq \mathcal{A}^{\mathbb{N}^{d}}$, let 
\begin{equation*}
	\left\{x|_{\mathbb{N}_{n_{1}\times n_{2}\times\cdots \times n_{d}(i_{1},\ldots,i_{d})}}:x\in X, n_{j}\geq 1, 1\leq j\leq d \text{ and }(i_{1},\ldots,i_{d})\in\mathbb{N}^{d} \right\}
\end{equation*}
be the set of block patterns. We recall that $X$ satisfies \emph{block gluing} with constant $N$ if for any two block patterns $u$ and $v$ on rectangular lattices $s(u)$ and $s(v)$ with $d(s(u),s(v))\geq N$, there exists $x\in X$ such that $x|_{s(u)}=u$ and $x|_{s(v)}=v$.
\begin{theorem}\label{thm Nd sft bd}
		If the $\mathbb{N}^d$ SFT $X$ satisfies block gluing with constant $N$, then for any $h\in \left[\frac{1}{i}\min_{ 1\leq k\leq d} h_i^{(k)}(X), \frac{1}{i}\max_{1\leq k \leq d} h_i^{(k)} (X)\right]$, there exist $0\leq i_1,...,i_d\leq 1$ such that $h$ can be realized by boundary complexity $h^{\bf f}$ where $f_k(n)=i_kn-i$ for all $1\leq k \leq d$. That is, $h^{\bf f}\left(X\right)=h=\sum_{k=1}^d\frac{1}{i}\frac{\prod_{\ell\neq k}i_{\ell}}{\sum_{k=1}^d\prod_{\ell\neq k}i_{\ell}} h_i^{(k)}$ with $\sum_{k=1}^d i_k=1$.
\end{theorem}

\begin{proof}
	For simplicity, we prove the case $d=3$, the case $d=2$ being a corollary and the cases $d>3$ being similar. Let $0\leq i_1,...,i_3 \leq 1$ with $i_1+i_2+i_3=1$, and the block gluing constant denoted by $N$. For any $n\in\mathbb{N}$, we first write the $i$-boundary $L(n)=[1,i_1n]\times [1,i_2 n]\times [1,i_3 n]\setminus [1,i_1n-i]\times [1,i_2 n-i]\times [1,i_3 n-i]$ as a union of the following 3 parts: 
	\begin{align*}
	&L_1(n): [i_1n-i+1,i_1n]\times [1,i_2 n]\times [1,i_3 n],\\
	&L_2(n): [1,i_1n]\times [i_2n-i+1,i_2 n]\times [1,i_3 n],\\
	&L_3(n):[1,i_1n]\times [1,i_2 n]\times [i_3n-i+1,i_3 n].
	\end{align*}

	Due to the block gluing constant $N$, we cut the $L_1(n)$, $L_2(n)$ and $L_3(n)$ into the following 3 rectangulars:
	\begin{align*}
	&L_1'(n):[i_1n-i+1,i_1n]\times [1,i_2 n-i-N]\times [1,i_3 n-i-N],\\
	&L_2'(n):[1,i_1n-i-N]\times [i_2n-i+1,i_2 n]\times [1,i_3 n-i-N],\\
	&L_3'(n):[1,i_1n-i-N]\times [1,i_2 n-i-N]\times [i_3n-i+1,i_3 n].
	\end{align*}
	
	Then we can glue all of the admissible local patterns in $L_1'(n)$, $L_2'(n)$ and $L_3'(n)$ by the following steps: Firstly, we glue the patterns in $L_1'(n)$ and $L_2'(n)$ to form an admissible local pattern in $[1,i_1n]\times[1,i_2 n]\times [1,i_3n-i-N]$; then we glue this admissible pattern and the pattern in $L_3'(n)$ to form an admissible pattern in $[1,i_1n]\times [1,i_2 n]\times [1,i_3 n]$.
	
	Thus, we have the following inequality,
		\begin{align*}
			&\left|\mathcal{P}\left(L_1'(n),X\right)\right|\left|\mathcal{P}\left(L_2'(n),X\right)\right|\left|\mathcal{P}\left(L_3'(n),X\right)\right|\\
			&\leq \left|\mathcal{P}\left(F_{\bf m}({\bf f}),X\right)\right|\\
			&\leq \left|\mathcal{P}\left(L_1(n),X\right)\right|\left|\mathcal{P}\left(L_2(n),X\right)\right|\left|\mathcal{P}\left(L_3(n),X\right)\right|,
		\end{align*}
	where $f_k(n)=i_kn-i$ for all $1\leq k \leq d$.
	
	This implies 
	\begin{align*}
		&\sum_{j=1}^3 \frac{1}{i}\frac{\log\left|\mathcal{P}\left(L_j'(n),X\right)\right|}{|L_j'(n)|/i}\frac{|L_j'(n)|}{|L(n)|}\\
		&=\frac{\log\left|\mathcal{P}\left(L_1'(n),X\right)\right|+\log\left|\mathcal{P}\left(L_2'(n),X\right)\right|\log\left|\mathcal{P}\left(L_3'(n),X\right)\right|}{|L(n)|}\\
	&\leq \frac{\log\left|\mathcal{P}\left(F_{\bf m}({\bf f}),X\right)\right|}{|L(n)|}\\
	&\leq \frac{\log\left|\mathcal{P}\left(L_1(n),X\right)\right|+\log\left|\mathcal{P}\left(L_2(n),X\right)\right|+\log\left|\mathcal{P}\left(L_3(n),X\right)\right|}{|L(n)|}\\
	&=\sum_{j=1}^3\frac{1}{i}\frac{\log\left|\mathcal{P}\left(L_j(n),X\right)\right|}{|L_j(n)|/i}\frac{|L_j(n)|}{|L(n)|}.		
	\end{align*}
	By taking $n\to \infty$, we have
	\[
	h^{\bf f}\left(X\right)=\frac{i_1i_2 h_i^{(3)}+i_2i_3 h_i^{(1)}+i_1i_3 h_i^{(2)}}{i\left(i_1i_2+i_2i_3+i_1i_3\right)}.
	\]
	The proof is thus completed.
\end{proof}	
	
Let $i\in \mathbb{N}$, and ${\bf H}_i(X),{\bf V}_i(X)$ be the associated transition matrices, namely, they are indexed by the vertical and horizontal patterns of length i and the transition rules from $H_i(X)$ and $V_i(X)$ respectively \cite{ban2005patterns,pavlov2012approximating}. Denote $\left[a;b\right]$ by $\left[a,b\right]$ interval if $a\leq b$ otherwise $\left[b,a\right]$.   	
	
\begin{corollary}\label{thm add constant hbc 2d}
	If the $\mathbb{N}^2$ SFT satisfies block gluing with constant $N$, then for any $h\in \left[ \frac{1}{i}\log\lambda_{{\bf H}_i(X)};\frac{1}{i}\log\lambda_{{\bf V}_i(X)}\right]$, there is a $0\leq t \leq 1$ such that $h$ can be realized by boundary complexity $h^{\bf f}$ where $f_1(n)=tn-i$ and $f_2(n)=(1-t)n-i$. That is, $h^{\bf f}\left(X \right)=h=\frac{1-t}{i}\log\lambda_{{\bf V}_i(X)}+\frac{t}{i}\log\lambda_{{\bf H}_i(X)}$.		
\end{corollary}
\begin{proof}
Since $h_i^{(1)}(X)=\log \lambda_{{\bf V}_i(X)}$, $h_i^{(2)}(X)=\log \lambda_{{\bf H}_i(X)}$
	and let 
	\begin{equation*}
		0\leq t=\frac{h-\frac{1}{i}\log\lambda_{{\bf V}_i}}{\frac{1}{i}\log\lambda_{{\bf H}_i}-\frac{1}{i}\log\lambda_{{\bf V}_i}} \leq 1.
	\end{equation*}
	The proof is complete by Theorem \ref{thm Nd sft bd}.
\end{proof}

\begin{remark}
For the equivalent condition of block gluing property see \cite{ban2021mixing}.
\end{remark}

\begin{example}For simplicity, we represent a $2\times 2$ pattern $\begin{array}{|c|c|}
	\hline
	a&b\\
	\hline
	c&d\\
	\hline
	\end{array}$ by $[a,b;c,d]$. 
	\item[1.] Let $\{0,1\}^{\mathbb{Z}_{2\times 2}}\setminus\mathcal{F}=\{[0,0;0,0],[1,0;0,0],[0,1;0,0],[0,0;1,0],[0,0;0,1]\}$, then the 1-boundary complexity is $\log g>h(X_{\mathcal{F}})$ by using the fact that $0$ is a safe symbol. 
	\item[2.] Let $\mathcal{F}=\{[0,1;0,1],[1,0;1,0],[0,1;1,1],[1,0;1,1],[1,1;0,1],[1,1;1,0],[1,1;1,1]\}$ be the forbidden set, it implies that the rule in the vertical (horizontal, resp.) direction is the same as the golden-mean shift (full shift, resp.). Then the 1-boundary complexites are of the form 
	\begin{equation*}
	\left(1-t\right)\log 2+t\log g, 0\leq t \leq 1,
	\end{equation*}
and any $h\in \left[h(X_{\mathcal{F}})=\log g, \log 2 \right]$ can be realized by $t=\frac{\log 2-h}{\log2-\log g}$.
	\item[3.] Let $\{0,1\}^{\mathbb{Z}_{2\times 2}}\setminus\mathcal{F}=\{[0,1;1,0],[1,0;0,1]\}$. This implies that the system contains only two global patterns, then the boundary complexites are $0=h(X_{\mathcal{F}})$.
	 	\item[4.] Let $\{0,1\}^{\mathbb{Z}_{2\times 2}}\setminus\mathcal{F}=\{[1,0;0,0],[0,1;0,0],[0,0;1,0],[0,0;0,1],[1,0;0,1]\}$. It can be easily checked that once the pattern on the top and rightmost strips of a rectangle is fixed, then the pattern on this rectangle with boundary are of these two strips is uniquely determined. Then the $i$-boundary complexites are
	 \begin{equation*}
	 	\lim_{m,n\to \infty}\frac{\log a_{m+n-1}}{mn-\left(m-i\right)\left(n-i\right)}=\frac{\log g}{i},
	 \end{equation*}
	 where $i\in\mathbb{N}$ and $a_{n+2}=a_{n+1}+a_{n}$ with $a_1=2$ and $a_2=3$.
	\item[5.] Let $\mathcal{F}=\{[1,0;0,0],[0,1;0,0],[0,0;1,0],[0,0;0,1],[0,1;1,1],[1,0;1,1],[1,1;0,1],[1,1;1,0]\}$ be the forbidden set. set. Then under the same argument as (4), we have the $i$-boundary complexites are 
\begin{equation*}
\lim_{m,n\to \infty}\frac{\log 2^{m+n-1}}{mn-\left(m-i\right)\left(n-i\right)}=\frac{\log2}{i},
\end{equation*}
where $i\in\mathbb{N}$.
\end{example}

\section{Surface entropy}
 Here we consider the surface entropy of 2-multiplicative integer system on $\mathbb{N}^d$, and we obtain the sub-exponential term. 
\begin{theorem}\label{theorem main}
	We have the following assertions.
\item[1.] For any two sequences $\{a_n\}_{n=1}^\infty$ and $\{b_n\}_{n=1}^\infty$ with $\lim_{n\to\infty}a_n=\infty$ and $\lim_{n\to\infty}b_n=\infty$, we have
\begin{equation*}
	\log \left|\mathcal{P}\left(\mathbb{Z}_{x_n,y_n},X^{\bf p}_{\Omega}\right)\right|-x_ny_nh=\left(1-\frac{1}{p_1p_2}\right)^2\sum_{i=r_n+1}^{\infty}\frac{x_ny_n}{\left(p_1p_2\right)^{i-1}}\log |\Omega_i|+o\left(a_nx_n+b_ny_n\right),
\end{equation*}	
where $r_n=\min\left\{ \left\lfloor \frac{\log x_n}{\log p_1} \right\rfloor,\left\lfloor \frac{\log y_n}{\log p_2} \right\rfloor\right\}$.
\item[2.] If $x_n=p_1^{k_1n}$ and $y_n=p_2^{k_2n}$ with $k_1\geq k_2 \in\mathbb{N}$, then we have
\begin{equation*}
	\log \left|\mathcal{P}\left(\mathbb{Z}_{x_n,y_n},X^{\bf p}_{\Omega}\right)\right|-x_ny_n h=-\left(1-\frac{1}{p_1p_2}\right)\log\left(\lambda_A\right) k_2n p_1^{\left(k_1-k_2\right)n}+o\left(k_2n p_1^{\left(k_1-k_2\right)n}\right).
\end{equation*}
\item[3.] If $x_n=y_n=p^n\pm k$ and $1\leq k \leq p$, then we have 
\begin{equation*}
\log \left|\mathcal{P}\left(\mathbb{Z}_{x_n,y_n},X^{\bf p}_{\Omega}\right)\right|-x_ny_nh=O\left(n\right).	
\end{equation*}
\end{theorem}
\begin{proof}
	\item[\bf1.] Since
	\begin{equation*}
		\begin{aligned}
		\log \left|\mathcal{P}\left(\mathbb{Z}_{x_n,y_n},X^{\bf p}_{\Omega}\right)\right|&=\sum_{i=1}^{r_n}\log|\Omega_i|^{\left|\mathcal{K}_{x_n\times y_n;i}\right|}\\
		&=\sum_{i=1}^{r_n}\frac{\left|\mathcal{K}_{x_n\times y_n;i}\right|}{\left|\mathcal{J}_{x_n\times y_n;i}\right|}\left|\mathcal{J}_{x_n\times y_n;i}\right|\log|\Omega_i|\\
		&=\sum_{i=1}^{r_n}\frac{\left|\mathcal{K}_{x_n\times y_n;i}\right|}{\left|\mathcal{J}_{x_n\times y_n;i}\right|}\left[\left|\mathcal{J}_{x_n\times y_n;i}\right|-\left(\frac{x_ny_n}{\left(p_1p_2\right)^{i-1}}-\frac{x_ny_n}{\left(p_1p_2\right)^i}\right)\right]\log|\Omega_i|\\
		&+\sum_{i=1}^{r_n}\left[\frac{\left|\mathcal{K}_{x_n\times y_n;i}\right|}{\left|\mathcal{J}_{x_n\times y_n;i}\right|}-\left(1-\frac{1}{p_1p_2}\right)\right]\left(\frac{x_ny_n}{\left(p_1p_2\right)^{i-1}}-\frac{x_ny_n}{\left(p_1p_2\right)^i}\right)\log|\Omega_i|\\
		&+\sum_{i=1}^{r_n}\left(1-\frac{1}{p_1p_2}\right)\left(\frac{x_ny_n}{\left(p_1p_2\right)^{i-1}}-\frac{x_ny_n}{\left(p_1p_2\right)^i}\right)\log|\Omega_i|.
				\end{aligned}
	\end{equation*}	
	Then by Lemma \ref{lemma decompose2}, we have
	\begin{equation*}
	\begin{aligned}
		&\log \left|\mathcal{P}\left(\mathbb{Z}_{x_n,y_n},X^{\bf p}_{\Omega}\right)\right|-x_ny_nh\\
		&=\sum_{i=1}^{r_n}\frac{\left|\mathcal{K}_{x_n\times y_n;i}\right|}{\left|\mathcal{J}_{x_n\times y_n;i}\right|}\left[\left(\left\lfloor \frac{x_n}{p_1^{i-1}}\right\rfloor\left\lfloor \frac{y_n}{p_2^{i-1}}\right\rfloor-\left\lfloor \frac{x_n}{p_1^i}\right\rfloor\left\lfloor \frac{y_n}{p_2^i}\right\rfloor\right)-\left(\frac{x_ny_n}{\left(p_1p_2\right)^{i-1}}-\frac{x_ny_n}{\left(p_1p_2\right)^i}\right)\right]\log|\Omega_i|\\
		&+\sum_{i=1}^{r_n}\left[\frac{\left|\mathcal{K}_{x_n\times y_n;i}\right|}{\left|\mathcal{J}_{x_n\times y_n;i}\right|}-\left(1-\frac{1}{p_1p_2}\right)\right]\left(\frac{x_ny_n}{\left(p_1p_2\right)^{i-1}}-\frac{x_ny_n}{\left(p_1p_2\right)^i}\right)\log|\Omega_i|\\
		&+\sum_{i=r_n+1}^{\infty}\left(1-\frac{1}{p_1p_2}\right)\left(\frac{x_ny_n}{\left(p_1p_2\right)^{i-1}}-\frac{x_ny_n}{\left(p_1p_2\right)^i}\right)\log|\Omega_i|.
	\end{aligned}
\end{equation*}	

For the first sum, we have
\begin{align*}
	&\left|\sum_{i=1}^{r_n}\frac{\left|\mathcal{K}_{x_n\times y_n;i}\right|}{\left|\mathcal{J}_{x_n\times y_n;i}\right|}\left[\left(\left\lfloor \frac{x_n}{p_1^{i-1}}\right\rfloor\left\lfloor \frac{y_n}{p_2^{i-1}}\right\rfloor-\left\lfloor \frac{x_n}{p_1^i}\right\rfloor\left\lfloor \frac{y_n}{p_2^i}\right\rfloor\right)-\left(\frac{x_ny_n}{\left(p_1p_2\right)^{i-1}}-\frac{x_ny_n}{\left(p_1p_2\right)^i}\right)\right]\log|\Omega_i|\right|\\
	&\leq \sum_{i=1}^{r_n}\left|\left(\left\lfloor \frac{x_n}{p_1^{i-1}}\right\rfloor\left\lfloor \frac{y_n}{p_2^{i-1}}\right\rfloor-\left\lfloor \frac{x_n}{p_1^i}\right\rfloor\left\lfloor \frac{y_n}{p_2^i}\right\rfloor\right)-\left(\frac{x_ny_n}{\left(p_1p_2\right)^{i-1}}-\frac{x_ny_n}{\left(p_1p_2\right)^i}\right)\right|\log|\Omega_i|\\
	&\leq \sum_{i=1}^{r_n}\left(\left|\left\lfloor \frac{x_n}{p_1^{i-1}}\right\rfloor\left\lfloor \frac{y_n}{p_2^{i-1}}\right\rfloor-\frac{x_ny_n}{\left(p_1p_2\right)^{i-1}}\right|+\left|\frac{x_ny_n}{\left(p_1p_2\right)^i}-\left\lfloor \frac{x_n}{p_1^i}\right\rfloor\left\lfloor \frac{y_n}{p_2^i}\right\rfloor\right|\right)\log|\Omega_i|\\
	&\leq \sum_{i=1}^{r_n}\left(\left|\left(\frac{x_n}{p_1^{i-1}}-1\right)\left(\frac{y_n}{p_2^{i-1}}-1\right)-\frac{x_ny_n}{\left(p_1p_2\right)^{i-1}}\right|+\left|\frac{x_ny_n}{\left(p_1p_2\right)^i}-\left( \frac{x_n}{p_1^i}-1\right)\left(\frac{y_n}{p_2^i}-1\right)\right|\right)\log|\Omega_i|\\
	&\leq\sum_{i=1}^{r_n}\left(\frac{x_n}{p_1^{i-1}}+\frac{y_n}{p_2^{i-1}}-1+\frac{x_n}{p_1^i}+\frac{y_n}{p_2^i}-1\right)\log r^i\\
	&=\left\{x_n\left(1+\frac{1}{p_1}\right)\left[\frac{\left(1-\frac{1}{p_1^{r_n}}\right)}{\left(1-\frac{1}{p_1}\right)^2}+\frac{r_n}{p_1^{r_n}}\right]+y_n\left(1+\frac{1}{p_2}\right)\left[\frac{\left(1-\frac{1}{p_2^{r_n}}\right)}{\left(1-\frac{1}{p_2}\right)^2}+\frac{r_n}{p_2^{r_n}}\right]-2r_n\right\}\log r\\
	&=O\left(x_n+y_n\right).		
\end{align*}

For the second sum, we have
\begin{align*}
&\left|\sum_{i=1}^{r_n}\left[\frac{\left|\mathcal{K}_{x_n\times y_n;i}\right|}{\left|\mathcal{J}_{x_n\times y_n;i}\right|}-\left(1-\frac{1}{p_1p_2}\right)\right]\left(\frac{x_ny_n}{\left(p_1p_2\right)^{i-1}}-\frac{x_ny_n}{\left(p_1p_2\right)^i}\right)\log|\Omega_i|\right|\\	
&\leq\sum_{i=1}^{r_n}\left|\frac{\left|\mathcal{K}_{x_n\times y_n;i}\right|}{\left|\mathcal{J}_{x_n\times y_n;i}\right|}-\left(1-\frac{1}{p_1p_2}\right)\right|\left(\frac{x_ny_n}{\left(p_1p_2\right)^{i-1}}-\frac{x_ny_n}{\left(p_1p_2\right)^i}\right)\log|\Omega_i|\\
&\leq\sum_{i=1}^{r_n}\frac{2\left(\frac{x_n}{p_1^{i-1}}+\frac{y_n}{p_2^{i-1}}\right)}{\left|\mathcal{J}_{x_n\times y_n;i}\right|}\left(\frac{x_ny_n}{\left(p_1p_2\right)^{i-1}}-\frac{x_ny_n}{\left(p_1p_2\right)^i}\right)\log|\Omega_i|\\
&\leq\sum_{i=1}^{r_n}\frac{2\left(\frac{x_n}{p_1^{i-1}}+\frac{y_n}{p_2^{i-1}}\right)}{\left|\mathcal{J}_{x_n\times y_n;i}\right|}\left[\left(\left\lfloor \frac{x_n}{p_1^{i-1}}\right\rfloor+1\right)\left(\left\lfloor \frac{y_n}{p_2^{i-1}}\right\rfloor+1\right)-\left\lfloor \frac{x_n}{p_1^i}\right\rfloor\left\lfloor \frac{y_n}{p_2^i}\right\rfloor\right]\log|\Omega_i|\\
&\leq 2\sum_{i=1}^{r_n}\left(\frac{x_n}{p_1^{i-1}}+\frac{y_n}{p_2^{i-1}}\right)\left[1+\frac{\left\lfloor \frac{x_n}{p_1^{i-1}}\right\rfloor+\left\lfloor \frac{y_n}{p_2^{i-1}}\right\rfloor+1}{\left|\mathcal{J}_{x_n\times y_n;i}\right|}\right]\log|\Omega_i|\\
&\leq 4\sum_{i=1}^{r_n}\left(\frac{x_n}{p_1^{i-1}}+\frac{y_n}{p_2^{i-1}}\right)\log r^i\\
&=O\left(x_n+y_n\right).
\end{align*}
\item[\bf2.] Since
\begin{equation*}
	\log \left|\mathcal{P}\left(\mathbb{Z}_{x_n,y_n},X^{\bf p}_{\Omega}\right)\right|=\left(1-\frac{1}{p_1p_2}\right)\left(p_1p_2-1\right)\sum_{i=1}^{r_n}\frac{x_ny_n}{\left(p_1p_2\right)^i}\log |\Omega_i|,
\end{equation*}	
and
\begin{equation*}
	x_ny_nh= \left(1-\frac{1}{p_1p_2}\right)\left(p_1p_2-1\right)\sum_{i=1}^\infty\frac{x_ny_n}{\left(p_1p_2\right)^i}\log |\Omega_i|.
\end{equation*}
Then
\begin{equation*}
	\begin{aligned}
		&\log \left|\mathcal{P}\left(\mathbb{Z}_{x_n,y_n},X^{\bf p}_{\Omega}\right)\right|-x_ny_n h\\
		&=-\left(1-\frac{1}{p_1p_2}\right)\left(p_1p_2-1\right)\sum_{i=r_n+1}^\infty\frac{x_ny_n}{\left(p_1p_2\right)^i}\log |\Omega_i|\\
		&=-\left(1-\frac{1}{p_1p_2}\right)\left(p_1p_2-1\right)\frac{x_ny_n}{\left(p_1p_2\right)^{r_n}}\sum_{i=1}^\infty\frac{1}{\left(p_1p_2\right)^i}\log |\Omega_{i+r_n}|.
	\end{aligned}
\end{equation*}

If $\Omega$ is an 1-$d$ mixing SFT with transition matrix $A$ and $\lambda_A$ is the largest eigenvalue of $A$ with normalized left and right eigenvectors $l$ and $r$ such that $l\cdot r=1$. By Theorem 4.5.12 in \cite{LM-1995}, we have
\begin{equation*}
	\left(A^n\right)_{ij}=\left[r_i l_j+p_{ij}\left(n\right) \right]\lambda_A^n, 
\end{equation*}
where $\lim_{n\to\infty}p_{ij}\left(n\right)=0$. Then 
\begin{equation*}
	\begin{aligned}
		&\lim_{n\to\infty}\frac{\log \left|\mathcal{P}\left(\mathbb{Z}_{x_n,y_n},X^{\bf p}_{\Omega}\right)\right|-x_ny_n h}{r_n\frac{x_ny_n}{\left(p_1p_2\right)^{r_n}}}\\
		&=-\lim_{n\to\infty}\frac{1}{r_n}\left(1-\frac{1}{p_1p_2}\right)\left(p_1p_2-1\right)\sum_{i=1}^\infty\frac{1}{\left(p_1p_2\right)^i}\log |\Omega_{i+r_n}|\\
		&=-\lim_{n\to\infty}\frac{1}{r_n}\left(1-\frac{1}{p_1p_2}\right)\left(p_1p_2-1\right)\sum_{i=1}^\infty\frac{1}{\left(p_1p_2\right)^i}\log \left|A^{i+r_n-1}\right|\\
		&=-\lim_{n\to\infty}\frac{1}{r_n}\left(1-\frac{1}{p_1p_2}\right)\left(p_1p_2-1\right)\sum_{i=1}^\infty\frac{\log\left(\sum_{k,j}\left[r_k l_j+p_{kj}\left(i+r_n-1\right) \right]\lambda_A^{i+r_n-1}\right)}{\left(p_1p_2\right)^{i}}\\
		&=-\lim_{n\to\infty}\frac{1}{r_n}\left(1-\frac{1}{p_1p_2}\right)\left(p_1p_2-1\right)\sum_{i=1}^\infty\frac{\log\left(\sum_{k,j}r_k l_j+p_{kj}\left(i+r_n-1\right)\right)+\log\left(\lambda_A^{i+r_n-1}\right)}{\left(p_1p_2\right)^{i}}\\
		&=-\lim_{n\to\infty}\frac{1}{r_n}\left(1-\frac{1}{p_1p_2}\right)\left(p_1p_2-1\right)\sum_{i=1}^\infty\frac{\log\left(\sum_k r_k\sum_j l_j+\sum_{k,j}p_{kj}\left(i+r_n-1\right)\right)+\left(i-1\right)\log\lambda_A}{\left(p_1p_2\right)^{i}}\\
		&-\left(1-\frac{1}{p_1p_2}\right)\left(p_1p_2-1\right)\sum_{i=1}^\infty\frac{\log\lambda_A}{\left(p_1p_2\right)^{i}}=-\left(1-\frac{1}{p_1p_2}\right)\log\lambda_A+\lim_{n\to\infty}\frac{1}{r_n}C_n\\
		&=-\left(1-\frac{1}{p_1p_2}\right)\log\lambda_A,
	\end{aligned}
\end{equation*}
where
\[
C_n=\left(1-\frac{1}{p_1p_2}\right)\left(p_1p_2-1\right)\sum_{i=1}^\infty\frac{\log\left(\sum_k r_k\sum_j l_j+\sum_{k,j}p_{kj}\left(i+r_n-1\right)\right)+\left(i-1\right)\log\lambda_A}{\left(p_1p_2\right)^{i}}<\infty,
\]
and
\[
\lim_{n\to \infty}C_n=C=\left(1-\frac{1}{p_1p_2}\right)\left(p_1p_2-1\right)\sum_{i=1}^\infty\frac{\log\left(\sum_k r_k\sum_j l_j\right)+\left(i-1\right)\log\lambda_A}{\left(p_1p_2\right)^{i}}<\infty.
\] 
\item[\bf3] If $x_n=y_n=p^n-k$ and $1\leq k\leq p$, then for $i=0, \left\lfloor\frac{x_n}{p^i}\right\rfloor=\left\lfloor\frac{y_n}{p^i}\right\rfloor=p^n-k$ and for $i>1,\left\lfloor\frac{x_n}{p^i}\right\rfloor=\left\lfloor\frac{y_n}{p^i}\right\rfloor=p^{n-i}-1$. Thus,
\begin{equation*}
		\left|\mathcal{J}_{x_n\times y_n;1}\right|=\left(p^n-k\right)^2-\left(p^{n-1}-1\right)^2=p^{2n}-2kp^n-p^{2n-2}+2p^{n-1}+k^2-1,
\end{equation*}
and for $i>1$, 
\begin{equation*}
		\left|\mathcal{J}_{x_n\times y_n;i}\right|=\left(p^{n-i+1}-1\right)^2-\left(p^{n-i}-1\right)^2=p^{2(n-i+1)}-2p^{n-i+1}-p^{2(n-i)}+2p^{n-i}.
\end{equation*}
Then
\begin{equation*}
	\begin{aligned}
		\left|\mathcal{K}_{x_n\times y_n;1}\right|&=\left|\mathcal{J}_{x_n\times y_n;1}\right|-\left[ \left(p^{n-1}-1\right)^2-\left(p^{n-2}-1\right)^2\right]\\
		&=p^{2n}-2p^{2n-2}+p^{2n-4}-2kp^{n}+4p^{n-1}-2p^{n-2}+k^2-1,
	\end{aligned}
\end{equation*}
and for $i>1$, 
\begin{equation*}
	\begin{aligned}
		\left|\mathcal{K}_{x_n\times y_n;i}\right|&=\left|\mathcal{J}_{x_n\times y_n;i}\right|-\left[ \left(p^{n-i}-1\right)^2-\left(p^{n-i-1}-1\right)^2\right]\\
		&=p^{2n-2i+2}-2p^{2n-2i}+p^{2n-2i-2}-2p^{n-i+1}+4p^{n-i}-2p^{n-i-1}.
	\end{aligned}
\end{equation*}
On the other hand, 
\begin{equation*}
	\begin{aligned}
		\left(1-\frac{1}{p^2}\right)\left(\frac{x_ny_n}{\left(p^2\right)^{i-1}}-\frac{x_ny_n}{\left(p^2\right)^i}\right)&=\frac{\left(p^2-2+\frac{1}{p^2}\right)\left(p^n-k\right)^2}{p^{2i}}\\
		&=p^{2n-2i+2}-2p^{2n-2i}+p^{2n-2i-2}\\
		&-2kp^{n-2i+2}+4kp^{n-2i}-2kp^{n-2i-2}\\
		&+\frac{k^2}{p^{2i}}\left(p^2-2+\frac{1}{p^2}\right)
	\end{aligned}
\end{equation*}
for all $1\leq i \leq r_n=n-1$. 

Hence for $i=1$
\begin{equation*}
	\begin{aligned}
		&\left|\mathcal{K}_{x_n\times y_n;1}\right|-	\left(1-\frac{1}{p^2}\right)\left(x_ny_n-\frac{x_ny_n}{p^2}\right)\\
		&=4p^{n-1}-\left(2+4k\right)p^{n-2}+2kp^{n-4}-\frac{k^2}{p^2}\left(-2+\frac{1}{p^2}\right)-1,
	\end{aligned}	
\end{equation*}
and for $i>1$,
\begin{equation*}
	\begin{aligned}
		&\left|\mathcal{K}_{x_n\times y_n;i}\right|-	\left(1-\frac{1}{p^2}\right)\left(\frac{x_ny_n}{\left(p^2\right)^{i-1}}-\frac{x_ny_n}{\left(p^2\right)^i}\right)\\
		&=\frac{p^n}{p^i}\left(-2p+4-\frac{2}{p}\right)-\frac{p^n}{p^{2i}}k\left(-2p^{2}+4-\frac{2}{p^2}\right)-\frac{k^2}{p^{2i}}\left(p^2-2+\frac{1}{p^2}\right).
	\end{aligned}	
\end{equation*}
If $\Omega$ is a full shift, then $\Omega_i=2^i$. Then we have
	\begin{align*}
		&\sum_{i=1}^{n-1}\left[\left|\mathcal{K}_{x_n\times y_n;i}\right|-	\left(1-\frac{1}{p^2}\right)\left(\frac{x_ny_n}{\left(p^2\right)^{i-1}}-\frac{x_ny_n}{\left(p^2\right)^i}\right)\right]\log |\Omega_i|\\
		&=\log 2\left[4p^{n-1}-\left(2+4k\right)p^{n-2}+2kp^{n-4}-\frac{k^2}{p^2}\left(-2+\frac{1}{p^2}\right)-1\right]\\
		&+\log 2\sum_{i=2}^{n-1}\frac{ip^n}{p^i}\left(-2p+4-\frac{2}{p}\right)-\frac{ikp^n}{p^{2i}}\left(-2p^{2}+4-\frac{2}{p^2}\right)-\frac{ik^2}{p^{2i}}\left(p^2-2+\frac{1}{p^2}\right)\\
		&=\log 2\left[4p^{n-1}-\left(2+4k\right)p^{n-2}+2kp^{n-4}-\frac{k^2}{p^2}\left(-2+\frac{1}{p^2}\right)-1\right]\\
		&+\log 2\left(-2p+4-\frac{2}{p}\right)\sum_{i=2}^{n-1}\frac{ip^n}{p^i}\\
		&-\log 2\left(-2p^2+4-\frac{2}{p^2}\right)k\sum_{i=2}^{n-1}\frac{ip^n}{p^{2i}}\\
		&-\log 2\left(p^2-2+\frac{1}{p^2}\right)k^2\sum_{i=2}^{n-1}\frac{i}{p^{2i}}.
	\end{align*}	
After computing, the above sum equals
	\begin{align*}		
	&\log 2\left[4p^{n-1}-\left(2+4k\right)p^{n-2}+2kp^{n-4}-\frac{k^2}{p^2}\left(-2+\frac{1}{p^2}\right)-1\right]\\
	&+\log 2\left(-2p+4-\frac{2}{p}\right)\left[\frac{p^{n-3}-1}{\left(1-\frac{1}{p}\right)^2}+\frac{2p^{n-2}-\left(n-1\right)}{1-\frac{1}{p}}\right]\\
		&-\log 2\left(-2p^2+4-\frac{2}{p^2}\right)k\left[\frac{p^{n-6}-\frac{1}{p^n}}{\left(1-\frac{1}{p^2}\right)^2}+\frac{2p^{n-4}-\frac{n-1}{p^n}}{1-\frac{1}{p^2}}\right]\\
		&-\log 2\left(p^2-2+\frac{1}{p^2}\right)k^2\left[\frac{\frac{1}{p^6}-\frac{1}{p^{2n}}}{\left(1-\frac{1}{p^2}\right)^2}+\frac{\frac{2}{p^4}-\frac{n-1}{p^{2n}}}{1-\frac{1}{p^2}}\right]\\
		&=\log 2\left[-\frac{k^2}{p^2}\left(-2+\frac{1}{p^2}\right)-1\right]\\
		&+\log 2\left(-2p+4-\frac{2}{p}\right)\left[\frac{-1}{\left(1-\frac{1}{p}\right)^2}+\frac{-\left(n-1\right)}{1-\frac{1}{p}}\right]\\
		&-\log 2\left(-2p^2+4-\frac{2}{p^2}\right)k\left[\frac{-\frac{1}{p^n}}{\left(1-\frac{1}{p^2}\right)^2}+\frac{-\frac{n-1}{p^n}}{1-\frac{1}{p^2}}\right]\\
		&-\log 2\left(p^2-2+\frac{1}{p^2}\right)k^2\left[\frac{\frac{1}{p^6}-\frac{1}{p^{2n}}}{\left(1-\frac{1}{p^2}\right)^2}+\frac{\frac{2}{p^4}-\frac{n-1}{p^{2n}}}{1-\frac{1}{p^2}}\right]\\
		&=O(n).
	\end{align*}	
Thus, $\log \left|\mathcal{P}\left(\mathbb{Z}_{x_n,y_n},X^{\bf p}_{\Omega}\right)\right|-x_ny_nh=O\left(n\right)$. The proof of the case $x_n=y_n=p^n+k$ is similar to $x_n=y_n=p^n-k$. Thus, the proof is complete.
\end{proof}	

\begin{example}
	Let $A=\left[\begin{matrix}
	1&1\\1&0
	\end{matrix}\right]$ and $p_1=2,p_2=3$. If $x_n=2^n$ and $y_n=3^n$, the numerical result is listed below. It should be noted that the rate of convergence is slow since it is the coefficient of a second order term. 
	\begin{center}
			\begin{tabular}{c|c}
			$n$&$\left|\frac{\log \left|\mathcal{P}\left(\mathbb{Z}_{x_n,y_n},X^{\bf p}_{\Omega}\right)\right|-x_ny_n h}{n}-\frac{-5\log g}{6}\right|$\\
			\hline
			$1$&$1.0189$\\
			$10$&$0.1014$\\
			$100$&$0.0102$\\
			$500$&$0.0020$\\
			$1000$&$0.0010$\\
			$1400$&$0.0007$\\
		\end{tabular}
	\end{center}
\end{example}

Theorem \ref{thm 2.1 nd} is a $\mathbb{N}^d$ version of Theorem \ref{theorem main}. 

\begin{theorem}\label{thm 2.1 nd}
		For $d\geq 1$, we have the following assertions.
	\item[1.] For any sequences $\{a_{n}^{(1)}\}_{n=1}^\infty,...,\{a_{n}^{(d)}\}_{n=1}^\infty$ with $a_{n}^{(i)}\to \infty$ as $n\to\infty$ for all $1\leq i \leq d$, we have
	\begin{align*}
		&\log \left|\mathcal{P}\left(\mathbb{Z}_{x_n^{(1)},...,x_n^{(d)}},X^{\bf p}_{\Omega}\right)\right|-\prod_{i=1}^d x_n^{(i)}h\\
		&=\left(1-\frac{1}{\prod_{i=1}^d p_i}\right)^2\sum_{j=r_n+1}^{\infty}\frac{\prod_{i=1}^d x_n^{(i)}}{\left(\prod_{i=1}^d p_i\right)^{j-1}}\log |\Omega_j|+\mathcal{G}\left(a_{1n},...,a_{dn},x_n^{(1)},...,x_n^{(d)}\right),
	\end{align*}	
	where $r_n=\min_{1\leq i\leq d} \left\lfloor \frac{\log x_n^{(i)}}{\log p_i} \right\rfloor$ and $\mathcal{G}\left(a_{n}^{(1)},...,a_{n}^{(d)},x_n^{(1)},...,x_n^{(d)}\right)=o\left(\sum_{i=1}^d \prod_{j\neq i}a_{n}^{(j)}x_n^{(j)}\right)$.
	\item[2.] If $x_n^{(i)}=p_i^{k_in}, k_i\in\mathbb{N}$ for all $1\leq i \leq d$, then 
	\begin{equation*}
		\log \left|\mathcal{P}\left(\mathbb{Z}_{x_n^{(1)},...,x_n^{(d)}},X^{\bf p}_{\Omega}\right)\right|-\prod_{i=1}^d x^{(i)}_nh=-\left(1-\frac{1}{\prod_{i=1}^d p_i}\right)\log\lambda_Ar_n\frac{\prod_{i=1}^dx^{(i)}_n}{\left(\prod_{i=1}^dp_i\right)^{r_n}}+\mathcal{G}\left(x_n^{(1)},...,x_n^{(d)}\right),
	\end{equation*}
	where $r_n=\min_{1\leq i\leq d}k_in,\lambda_A$ is the largest eigenvalue of $A$ and $\mathcal{G}\left(x_n^{(1)},...,x_n^{(d)}\right)=o\left(r_n\frac{\prod_{i=1}^dx^{(i)}_n}{\left(\prod_{i=1}^dp_i\right)^{r_n}}\right)$.
	\item[3.] If $x_n^{(i)}=p^n-k, 1\leq k \leq p$ for all $1\leq i \leq d$, then 
	\begin{equation*}
		\log \left|\mathcal{P}\left(\mathbb{Z}_{x_n^{(1)},...,x_n^{(d)}},X^{\bf p}_{\Omega}\right)\right|-\prod_{i=1}^dx^{(i)}_nh=O\left(n\right).	
	\end{equation*}
\end{theorem}

\begin{proof}
	The proofs are similar to Theorem \ref{theorem main} by replacing $p_1p_2$ and $x_ny_n$ with $\prod_{i=1}^dp_i$ and $\prod_{i=1}^dx^{(i)}_n$ respectively.
\end{proof}	

\begin{remark}
	We remark that in additive systems, \cite{pace2018surface} shows
	\begin{equation*}
		\log \left|\mathcal{P}\left(\mathbb{Z}_{n},\Sigma_A\right)\right|-nh=\log \left(\frac{\sum_k r_k\sum_j l_j}{\lambda_A}\right).
	\end{equation*}
	In multiplicative integer systems, when $d=1$, we have 
	\begin{equation*}
		\log \left|\mathcal{P}\left(\mathbb{Z}_{p^{kn}},X^p_{\Sigma_A}\right)\right|-p^{kn} h=-\left(1-\frac{1}{p}\right)\log\lambda_A kn+o\left(kn\right).
	\end{equation*}
\end{remark}

\bibliographystyle{amsplain}
\bibliography{ban}

\end{document}